\documentclass[11pt]{article}

\usepackage{latexsym,epsf}
\usepackage[dvips]{graphicx}
\usepackage{a4wide,amsthm,amssymb,amsfonts,epsfig,color,amsmath}
\usepackage[ansinew]{inputenc}

\theoremstyle{plain}
\newtheorem{theorem}{Theorem}
\newtheorem{lemma}[theorem]{Lemma}

\theoremstyle{definition}
\newtheorem{definition}{Definition}
\newtheorem{remark}{Remark}

\title{
The number of generalized balanced lines\footnote{This work started
at the 6th Iberian Workshop on Computational Geometry, in
Aveiro, and was concluded while Gelasio Salazar was visiting
Departament of Mathematics of Alcal\'a University under the program
Giner de los Ríos.} }

\author{
David Orden\thanks{Departamento de Matemáticas, Universidad de Alcalá, {\tt
david.orden@uah.es},
partially supported by grant
MTM2008-04699-C03-02.} \and Pedro Ramos\thanks{Departamento de Matemáticas, Universidad de Alcalá,
{\tt pedro.ramos@uah.es},
partially supported by grant
MTM2008-04699-C03-02.}
\and Gelasio Salazar
\thanks{Instituto de Física, Universidad Autónoma de San Luis Potosí, Mexico,
{\tt gsalazar@dec1.ifisica.uaslp.mx}}
}

\begin{document}

\date{}
\maketitle

\begin{abstract}
Let $S$ be a set of $r$ red points and $b=r+2\delta$ blue points in
general position in the plane, with $\delta\ge 0$. A line $\ell$
determined by them is {\em balanced} if in each open half-plane bounded by $\ell$ the
difference between the number of blue points and red points is
$\delta$. We show that every set $S$ as above has at least $r$
balanced lines. The main techniques in the proof are rotations and a
generalization, sliding rotations, introduced here.
\end{abstract}

\section{Introduction}
\label{section:introduction}

Let~$B$ and $R$ be, respectively, sets of blue and red points in the
plane, and let $S=B\cup R$ be in general position. Let $r=|R|$ and
$b=|B|=r+2\delta$, with $\delta \ge 0$. Furthermore, we are given
weights $\omega(p)=+1$ for $p\in B$ and $\omega(q)=-1$ for $q\in
R$. Given a halfplane~$H$, its weight is then defined as
$\omega(H)=\sum_{s\in S\cap H}\omega(s)$. Here and throughout this
paper, halfplanes are open unless otherwise stated.

\begin{definition}
A line~$\ell$ determined by two points of~$S$ is
\emph{balanced} if the two halfplanes it defines have weight
$\delta$. Observe that this implies that the two points of $S$
spanning $\ell$ have different 
colors.
\end{definition}

For $\delta=0$, we obtain the original result, as conjectured by George Baloglou,
and proved by Pach and Pinchasi via circular sequences:

\begin{theorem}[\cite{pach-pinchasi}]
\label{thm:pach-pinchasi}
Let $|R|=|B|=n$. Every set~$S$ as above determines at
least~$n$ balanced lines. This bound is tight.
\end{theorem}

Tightness is shown, e.g., by placing~$S$ on a convex $2n$-gon in such a
way that~$R$ is separated from~$B$ by a straight line.

The general result was proved by Sharir and Welzl in
an indirect manner, via an equivalence with a very special case of
the Generalized Lower Bound Theorem. This motivated them to ask for
a more direct and simpler proof.

\begin{theorem}[\cite{sharir-welzl}]
\label{thm:sharir-welzl} Let~$B$ and $R$ be, respectively, sets of
blue and red points in the plane, and let $S=B\cup R$ be in general
position. Let $r=|R|$ and $b=|B|=r+2\delta$, with $\delta \ge 0$. The
number of lines that pass through a point in~$B$ and a point in~$R$,
and such that the two induced halfplanes have weight $\delta$ is at
least $r$.  This number is attained if~$R$ and~$B$ can be separated by
a line.
\end{theorem}

In this paper we give a simple proof of Theorem~\ref{thm:sharir-welzl}
using elementary geometric techniques. Therefore, via the results in~\cite{sharir-welzl},
we provide a geometric proof of the following
very special case of the Generalized Lower Bound Theorem:

Let $\mathcal{P}$ be a convex polytope which is the intersection of $d+4$
halfspaces in general position in $\mathbb{R}^d$. Let its edges be oriented
according to a generic linear function (edges are directed from smaller to larger
value; ``generic'' means that the function evaluates to distinct values at the
vertices of $\mathcal{P}$).

\begin{theorem}[\cite{sharir-welzl}]
The number of vertices with $\lceil\tfrac{d}{2}\rceil - 1$ outgoing edges is at most
the number of vertices with $\lceil\tfrac{d}{2}\rceil$ outgoing edges.
\end{theorem}

All proofs in this paper can be easily translated
to the more general setting of circular sequences (see~\cite{orsarxiv}).

\section{Geometric tools}
\label{section:first_tool}

We assume that coordinate axes are chosen in such a way that all points
have different abscissa. The tools we use are inspired in the rotational movement introduced by
Erd\H{o}s et al.~\cite{elss}.

\begin{definition}
Let $P\subseteq S$. A \emph{$P^k$-rotation} is a family of directed lines $P^k_t$, where $t\in[0,2\pi]$
is the angle measured from the vertical axis,
defined as follows: $P^k_0$ contains a single point of~$P$, and as $t$ increases, it rotates
counterclockwise in such a way that
\begin{itemize}
\item[(i)] $|P\cap P^k_t|=1$ except for a finite number of events,
when $|P\cap P^k_t|=2$; and
\item[(ii)] whenever $|P\cap P^k_t|=1$, there are exactly $k$ points of~$P$ to the right of~$P^k_t$.
\end{itemize}
The common point $P\cap P^k_t=\{p\}$ is called the \emph{pivot}, and it changes precisely when $|P\cap P^k_t|=2$.
Observe that $P^k_0=P^k_{2\pi}$.
\end{definition}

\begin{definition}\label{def:sign_preserving}
Let $\ell^+$ and $\ell^-$ denote, respectively, the halfplanes to
the right and to the left of $\ell$. Let $\omega(\ell)$ be the weight
of $\ell^+$. Given a $P^k$-rotation, we say that
$P^k\geq \delta$ if $\omega({P^k_t})\geq \delta$ for every $t\in [0,2\pi]$,
and similarly for the rest of inequalities. A rotation $B^k$ is
\emph{$\delta$-preserving} if either $B^k\geq \delta$ or $B^k< \delta$.
Symmetrically, $R^k$ is \emph{$\delta$-preserving} if either $R^k\leq \delta$
or $R^k > \delta$.
\end{definition}

\begin{lemma}\label{l:0-1}
In an $R^k$-rotation, transitions $\delta\rightsquigarrow \delta+1$
and $\delta+1\rightsquigarrow \delta$ in $\omega({R^k_t})$ are
always through a balanced line. In a $B^k$-rotation, transitions
$\delta\rightsquigarrow \delta-1$ and $\delta-1\rightsquigarrow
\delta$ in $\omega({B^k_t})$ are always through a balanced line.
\end{lemma}

\begin{proof}
When a red point is found during an $R^k$-rotation, the weight of
the halfplane is preserved because the pivot point changes.
Therefore, the change $\delta\rightsquigarrow \delta+1$ happens when
a blue point is found in the head of~$R^k_t$
(Figure~\ref{fig:B_k_rotation}, left), while
$\delta+1\rightsquigarrow \delta$ happens when a blue point is found
in the tail of~$R^k_t$ (Figure~\ref{fig:B_k_rotation}, right). In
both cases, the points define a balanced line. For a $B^k$-rotation,
the proof is identical.
\end{proof}

\begin{figure}[htb]
    \begin{center}
    \includegraphics{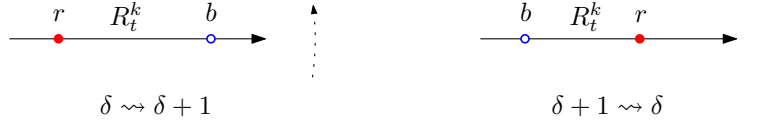}
    \caption{Transitions in an $R^k$-rotation are always through a balanced line.}
    \label{fig:B_k_rotation}
    \end{center}
\end{figure}

Claim 8.1 in \cite{pach-pinchasi} has now a more direct proof:
\begin{lemma}\label{l:halving}
If $r$ is odd, there exists a balanced line which is a halving line of $S$.
\end{lemma}

\begin{proof}
Let $k=\lfloor\tfrac{r}{2}\rfloor$ and consider an $R^k$-rotation.
If $R^k_0\leq\delta$, then $R^k_{\pi} > \delta$, and conversely.
Therefore, there exist transitions $\delta\rightsquigarrow \delta+1$
and $\delta+1\rightsquigarrow \delta$ in $\omega({R^k_t})$ which,
from Lemma~\ref{l:0-1} are always through a balanced line. Observe
that both transitions are through the same balanced line, with
angles $t_0$ and $t_0+\pi$.
\end{proof}

\begin{remark}\label{r:rem1}
Let us observe that Theorem~1.4 in~\cite{pach-pinchasi}, which states that Theorem~\ref{thm:pach-pinchasi}
is true when~$R$ and~$B$ are separated by a line~$\ell$, has now an easier proof:
if we start $R^k$-rotations with a line parallel to $\ell$, for each $k$ there exist exactly one
transition $\delta\rightsquigarrow \delta+1$ and one transition $\delta+1\rightsquigarrow \delta$
which, from Lemma~\ref{l:0-1}, correspond always to a balanced line.
If $r$ is even, there are 2 balanced lines for $k=0,\ldots,\tfrac{r}{2}-1$, for a total
of $r$ balanced lines, while if $r$ is odd there are 2 balanced lines
for $k=0,\ldots,\lfloor\tfrac{r}{2}\rfloor-1$ and 1 balanced line for
$k=\lfloor\tfrac{r}{2}\rfloor$.
\end{remark}

\begin{remark}
Lemmas~\ref{l:0-1} and~\ref{l:halving} conclude the proof of
Theorem~\ref{thm:pach-pinchasi} if no $R^k$-rotation is $\delta$-preserving
or if no $B^k$-rotation (with $k\geq\delta)$ is $\delta$-preserving.
Hence, in the following we assume that there exists either at least one $R^k$-rotation
or one $B^k$-rotation (with $k\geq\delta)$ which
is $\delta$-preserving.
\end{remark}

\begin{lemma}\label{l:BR}
Let $0\leq j \leq \lfloor\tfrac{r}{2}\rfloor$. If $R^j> \delta$ then
$B^{j+\delta}\geq \delta$, while if $B^{j+\delta}<\delta$ then $R^j \leq \delta$.
\end{lemma}

\begin{proof}
  Consider the line $R^j_{t_0}$. The halfplane $(R^j_{t_0})^+$ contains $j$ red points and
  $b>j+\delta$ blue points. Therefore, the line $B^{j+\delta}_{t_0}$ is to the right
  of $(R^j_{t_0})^+$ and contains at most $j$ red points. Then, $\omega(B^{j+\delta}_{t_0})\geq \delta$.
  The proof of the second statement is analogous.
\end{proof}

The next definition generalizes the concept of $P^k$-rotation in two different ways:
parallel movements are permitted and the number of points to the right of the line can change.

\begin{definition}
\label{def:sliding_rotation}
A \emph{$P$-sliding rotation} consists in moving a directed line~$\ell$
continuously, starting with an $\ell_0$ which contains a single point $p_0\in P$,
and composing rotation around a point of $P$ (the pivot) and parallel displacement
(in either direction) until the next point of~$P$ is found.
Furthermore, after a $2\pi$ rotation is
completed, the line $\ell_0$ must be reached again.
\end{definition}

This movement is clearly a continuous curve in the space of lines in the plane.
For instance, if a line is parameterized as a point in $S^1\times\mathbb{R}$, a
$P$-sliding rotation describes a (non-strictly) angular-wise monotone curve, with
vertical segments corresponding to parallel displacements.

Let $\Sigma$ be a $P$-sliding rotation.
Let us denote by $\Sigma_t$ the line with
angle $t$ with respect to the vertical axis defined as follows: if there is no parallel displacement
at angle~$t$, then $\Sigma_t$ denotes the corresponding line. Otherwise, it denotes
the leftmost line corresponding to angle~$t$.

\begin{definition}
A $P$-sliding rotation $\Sigma$ is \emph{positively oriented} if $\Sigma_{t+\pi}$
 is to the left of~$\Sigma_t$ for all~$t\in[0,\pi)$.
\end{definition}

That $\Sigma\geq \delta$, as well as the rest of inequalities, is
defined exactly as in Definition~\ref{def:sign_preserving}.
Similarly, a $B$-sliding rotation $\Sigma$ is $\delta$-preserving if
$\Sigma\geq \delta$, while an $R$-sliding rotation is
$\delta$-preserving if $\Sigma\leq \delta$. The following definition
is the crux of the rest of the paper.

\begin{definition}
\label{def:waist-gamma}
Let~$\mathcal{S}$ be the set of all positively oriented, $\delta$-preserving $B$-sliding rotations
 and $R$-sliding rotations.
The \emph{waist} of a $P$-sliding rotation $\Sigma\in\mathcal{S}$ is
$$
\min_{t\in [0,\pi]} |P\cap \Sigma_t^- \cap \Sigma_{t+\pi}^-|.
$$
We denote by~$\Gamma$ the sliding rotation of~$\mathcal{S}$ with the smallest waist.
\end{definition}

Note that the set~$\mathcal{S}$ is non-empty because we have assumed that there exist
$\delta$-preserving $B^k$- or $R^k$-rotations, which are a particular type of
sliding rotations. Furthermore, the waist takes only a finite number of values,
so it has a minimum. If the minimum is
not unique, we can pick any of the sliding rotations achieving it.

\section{Main result}
\label{sec:main_result}

Assume that $\Gamma$ is a $\delta$-preserving $R$-sliding rotation (i.e. $\Gamma\leq\delta$).
In this case, we will manage to prove that there exist at least $r$ balanced lines.
For the case of $\Gamma$ being a $\delta$-preserving $B$-sliding rotation, the
same arguments would show that there exist at least $b$ balanced lines.

\begin{lemma}
\label{lemma:F+H} Let $\Gamma_0$ and $\Gamma_{\pi}$ be the lines
achieving the waist of $\Gamma$, let $\overline{\Gamma}_0^+$ be the
closed halfplane to the right of $\Gamma_0$ and let
$F=R\cap\overline{\Gamma}_0^+$. For every $k\in\{0,\ldots,|F|-1\}$,
during an $F^k$-rotation a balanced line is found. Similarly, let
$H=R\cap\overline{\Gamma}_{\pi}^+$. For every
$k\in\{0,\ldots,|H|-1\}$, during an $H^k$-rotation a balanced line
is found.
\end{lemma}

\begin{proof}
Figure~\ref{fig:fk} illustrates the situation.
On the one hand, $F^k_0$ is to the right of~$\Gamma_0$ and,
since~$\Gamma$ is positively oriented, $F^k_{\pi}$ is to the left of~$\Gamma_{\pi}$.
This implies that there is
a~$t_1\in[0,\pi]$ such that $F^k_{t_1}=\Gamma_{t_1}$ and therefore
$\omega({F^k_{t_1}})\leq \delta$. On the other hand, $F^k_0$ is to the
left of~$\Gamma_{\pi}$ and $F^k_{\pi}$ is to the right
of~$\Gamma_0$, therefore, there exists a~$t_2\in[0,\pi]$ such that
$F^k_{t_2}$ and $\Gamma_{t_2+\pi}$ are the same line with opposite
directions. Since $\omega(\Gamma_{t_2+\pi})\leq \delta$, then
$\omega({F^k_{t_2}})\geq \delta$. If $\omega(\Gamma_{t_2+\pi}) = \delta$
and the line contains a blue point, then it is a balanced line found
in a transition $\delta\rightsquigarrow \delta +1$. Otherwise,
$\omega({F^k_{t_2}})>\delta$ and hence a transition
$\delta\rightsquigarrow \delta+1$ has occurred for a $t\in (t_1,t_2)$.

Now, observe that $R\smallsetminus F\subset {\Gamma_0}^-$. Hence, in the
$F^k$-rotation for~$t\in[0,\pi]$, all the points in~$R\smallsetminus F$
are found by the head of the line. This implies that a change
$\delta\rightsquigarrow\delta+1$ in the weight of the right halfplane can only
occur when a blue point is found in the head of the ray (as in
Figure~\ref{fig:B_k_rotation}, left), hence defining a balanced line. The
proof for~$H$ is identical. 
\end{proof}

\begin{figure}
    \begin{center}
    \includegraphics{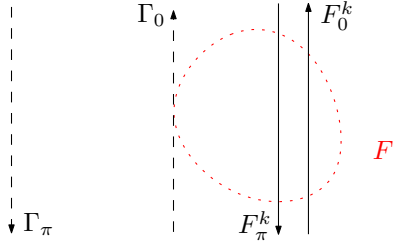}
    \caption{Illustration of the proof of Lemma~\ref{lemma:F+H}.}
    \label{fig:fk}
    \end{center}
\end{figure}

Before moving on, let us point out that the~$|F|+|H|$ balanced lines given by
Lemma~\ref{lemma:F+H} are different, because they have exactly~$k$ points of~$F$,
respectively~$H$, to the right. Let now $C^{\Gamma}_t$ be the \emph{central region}
defined by the sliding rotation~$\Gamma$ at instant~$t$, defined as $C^{\Gamma}_t=\Gamma_t^- \cap \Gamma_{t+\pi}^-$.
Observe that, for the corresponding~$t$, the transitions $\delta\rightsquigarrow\delta+1$ in the proof of
Lemma~\ref{lemma:F+H} correspond to balanced lines inside or in the boundary of the central region.

\begin{lemma}
\label{lemma:sign_change_G_k} Let $G=R\smallsetminus(F\cup H)$. For
$k\in\{0,\ldots, \lceil |G|/2 \rceil -1\}$, every $G^k$-rotation has
at least two transitions between $\delta$ and $\delta+1$, which
correspond to lines inside or in the boundary of the central
region., i.e., for the corresponding~$t$, $G^k_t\in C^{\Gamma}_t$.
\end{lemma}

\begin{proof}
Let us consider first the case when $r$ is odd and $k=\lfloor |G|/2\rfloor$.
$G^k_0$ and $G^k_{\pi}$ are the same line with opposite
directions. Therefore, if $\omega(G^k_{0})\leq \delta$ then $\omega(G^k_{\pi})> \delta$ and there
must be at least two transitions as stated. These transitions correspond to lines in
the central region because $\Gamma$ is positively oriented.

For the rest of cases, observe that, by construction, $G^k_{0}\in C^{\Gamma}_0$.
According to the value of $\omega(G^k_{0})$, we distinguish two cases:
\begin{itemize}

  \item $\omega(G^k_{0})\leq\delta$. If there exist some values for which $G^k_{t}=\Gamma_{t}$,
  let $t_1$ and $t_2$ be, respectively, the minimum and maximum of them. If there is
  no such value, take $t_1=t_2=2\pi$.
  If $G^k$ takes the value $\delta+1$ in the interval $(0,t_1)$ it must have transitions
  $\delta\rightsquigarrow\delta+1$ and $\delta+1\rightsquigarrow\delta$, and the same
  is true for $(t_2,2\pi)$. Finally, observe that $G^k$ must take the value $\delta+1$
  at least once,
  because in other case the sliding rotation obtained by concatenating $G^k$ in $(0,t_1)$,
  $\Gamma$ in $(t_1,t_2)$ and $G^k$ in $(t_2,2\pi)$ would be a $\delta$-preserving sliding rotation
  of waist smaller than the waist of $\Gamma$.

  \item $\omega(G^k_{0}) > \delta$. If there exist some values for which $G^k_{t}=\Gamma_{t}$,
  let $t_1$ and $t_2$ be, respectively, the minimum and maximum of them. $G^k_{t}$
  takes the value $\delta$ in the intervals $(0,t_1)$ and $(t_2,2\pi)$ and therefore
  the lemma follows.
  In other case, if $G^k_t$ takes the value $\delta$ in the central region, it must
  have also transition $\delta\rightsquigarrow\delta+1$.
  Finally, if $\omega(G^k_t) > \delta$ for all $t\in[0,2\pi]$  we could construct a sliding rotation~$\Sigma$
  contradicting the choice of~$\Gamma$: for each~$t$, consider as $\Sigma_t$ the parallel to~$G^k_t$
  which passes through the first blue point to the right of~$G^k_t$. It is easy to see that
  $\Sigma_t\geq\delta$, because between $\Gamma_t$ and $G^k_t$ there are always at least two blue points.

\end{itemize}
\vspace{-2.em}
\end{proof}

The following lemma, which already appeared as Claim~6.4 in~\cite{pach-pinchasi},
will be enough to conclude the proof of Theorem~\ref{thm:pach-pinchasi}.

\begin{lemma}
\label{lemma:balanced_or_recharge_G_k}
Transitions $\delta\rightsquigarrow \delta+1$ and $\delta+1\rightsquigarrow \delta$ in a $G^k$-rotation
are always either a balanced line or a
$\delta+1\rightsquigarrow \delta$ transition in an $F^j$-rotation, $j\in\{0,\ldots,|F|-1\}$ or an
$H^j$-rotation, $j\in\{0,\ldots,|H|-1\}$.
\end{lemma}

\begin{proof}
On the one hand, a balanced line is achieved if there is such a
transition because a blue point is found. See
Figure~\ref{fig:B_k_rotation}. On the other hand, if the point
inducing the transition is~$r\in R$, then necessarily~$r\in
R\smallsetminus G$ (since the $G^k$-rotation changes pivot whenever
a point of~$G$ is found). Figure~\ref{fig:gk-f} illustrates that a
$\delta+1\rightsquigarrow \delta$ transition appears for an
$F^j$-rotation with pivot~$g$, both if~$f\in F$ is found in the tail
(left picture) or if~$f\in F$ is found in the head (right picture).
\begin{figure}[htb]
    \begin{center}
    \includegraphics{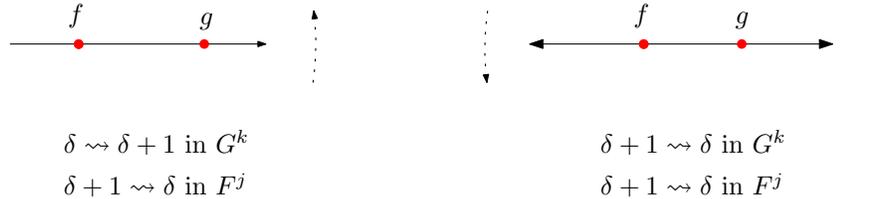}
    \caption{Transitions when a point~$f\in F\subset R$ found in a $G^k$-rotation induces a
    $\delta+1\rightsquigarrow \delta$ transition in an $F^j$-rotation.}
    \label{fig:gk-f}
    \end{center}
\end{figure}
Note that in the right picture the weight of both halfplanes is~$\delta+1$. The case in which the point
found is~$h\in H$ works similarly.
\end{proof}

The following simple observations show that the number of balanced lines is at least~$r$ which,
together with Remark~\ref{r:rem1}, finishes the proof of Theorem~\ref{thm:sharir-welzl}:

\begin{itemize}
\item[i)] Lemma~\ref{lemma:F+H} gives~$|F|+|H|$ different balanced lines.

\item[ii)] Lemmas~\ref{lemma:sign_change_G_k}
and~\ref{lemma:balanced_or_recharge_G_k} give~$|G|$ lines which are, either a balanced line,
or a $\delta+1\rightsquigarrow \delta$ transition at the central region for an~$F^j$-~or~$H^j$-rotation.

\item[iii)] Each transition in ii) forces a new $\delta\rightsquigarrow \delta+1$ transition at the central region for
an~$F^j$-~or~$H^j$-rotation which correspond, as in the proof of Lemma~\ref{lemma:F+H},
to a new balanced line.
\end{itemize}



\section{Acknowledgements}

The authors thank Jesús García for helpful discussions.

\end{document}